\documentclass [11pt]{article}
\usepackage{amssymb,amsmath,amsthm,comment}

\def\N{\mathbb{N}}

\def\F{\mathbb{F}}

\newtheorem{theorem}{Theorem}[section]
\newtheorem{proposition}[theorem]{Proposition}
\newtheorem{corollary}[theorem]{Corollary}
\newtheorem{lemma}[theorem]{Lemma}

\begin{document}
\title{All bi-unitary perfect polynomials over $\F_2$ with at most four irreducible factors}
 \author{Olivier Rahavandrainy \\
Univ Brest, UMR CNRS 6205\\
Laboratoire de Math\'ematiques de Bretagne Atlantique\\
e-mail : olivier.rahavandrainy@univ-brest.fr}
\maketitle
\newpage~\\
{\bf{Abstract}}\\
We give, in this paper, all bi-unitary perfect polynomials
over the prime field~$\F_2$, with at most four irreducible factors.

{\section{Introduction}}
Let $S \in \F_2[x]$ be a nonzero polynomial. We say that $S$ is odd if $\gcd(S,x(x+1))=1$, $S$ is even
if it is not odd.
A {\it{Mersenne} $(prime)$} is a polynomial (irreducible) of the form $1+x^a(x+1)^b$, with $\gcd(a,b) = 1$.
A divisor $D$ of $S$ is called unitary if $\gcd(D,S/D)=1$. We denote by $\gcd_u(S,T)$ the greatest
common unitary divisor of $S$ and $T$.
A divisor $D$ of $S$ is called bi-unitary if $\gcd_u(D,S/D)=1$.\\
We denote by $\sigma(S)$ (resp. $\sigma^*(S)$, $\sigma^{**}(S)$) the sum of all divisors
(resp. unitary divisors, bi-unitary divisors)
of $S$.
The functions $\sigma$, $\sigma^*$ and $\sigma^{**}$ are all multiplicative.
We say that a polynomial $S$ is \emph{perfect} (resp. \emph{unitary perfect},
\emph{bi-unitary perfect}) if $\sigma(S) = S$
(resp. $\sigma^*(S)=S$, $\sigma^{**}(S)=S$). \\
Finally, we say that $S$ is \emph{indecomposable bi-unitary perfect} ($i.b.u.p.$)
if it is bi-unitary perfect but it is not a product of two coprime
nonconstant bi-unitary perfect polynomials.\\
As usual, $\omega(S)$ designates the number of distinct irreducible
factors of $S$.\\
Several studies are done about perfect and unitary perfect.
In particular, we gave (\cite{Gall-Rahav7}, \cite{Gall-Rahav5}, \cite{Gall-Rahav11}) the list of
all (unitary) perfect polynomials $A$ over $\F_2$ (even or not),
with $\omega(A) \leq 4$. Recently, we also get more results in \cite{Gall-Rahav14}.

The case where $A$ has only Mersenne prime as odd divisors is already treated in \cite{Gall-Rahav15}.

We are interested in bi-unitary perfect polynomials (b.u.p. polynomials) $A$
with $\omega(A) \leq 4$.
If $A \in \F_2[x]$ is nonconstant b.u.p., then $x(x+1)$ divides $A$ so that $\omega(A) \geq 2$
(see Lemma \ref{nombreminimal}). Moreover, the only
b.u.p. polynomials over $\F_2$ with exactly two prime
factors are $x^2(x+1)^{2}$ and $x^{2^n-1}(x+1)^{2^n-1}$, for any nonnegative integer $n$
(\cite{BeardBiunitary}, Theorem 5).
We prove (Theorems \ref{caseomega3} and \ref{caseomega4}) that the only
b.u.p. polynomials $A \in \F_2$, with $\omega(A) \in \{3, 4\}$, are those given
in \cite{BeardBiunitary}, plus four other ones. Note that all odd irreducible divisors of the $C_j$'s are Mersenne primes (there is a misprint for $C_6$, in {\rm{\cite{BeardBiunitary}}).

In the rest of the paper, for $S \in \F_2[x]$, we denote by
$\overline{S}$ the polynomial obtained from $S$ with $x$ replaced by
$x+1$: $\overline{S}(x) = S(x+1)$.\\
As usual, $\N$ (resp. $\N\sp{*}$) denotes the set of nonnegative
integers (resp. of positive integers).\\
For $S, T \in \F_2[x]$ and $n \in \N^*$, we write: $S^n\| T$ if $S^n | T$ but $S^{n+1} \nmid T$.\\
Finally, let ${\mathcal{M}}$ denotes the set of all Mersenne primes.\\
We consider the following polynomials over $\F_2$:
$$\begin{array}{l}
M_1=1+x+x^2 = \sigma(x^2),\ M_2=1+x+x^3,\ M_3=\overline{M_2}=1+x^2+x^3,\\
M_4=1+x+x^2+x^3+x^4 = \sigma(x^4), M_5=\overline{M_4}=1+x^3+x^4,\\
S_1=1+x(x+1)M_1 = 1+x+x^4,\\
C_1 =x^3(x+1)^4M_1, C_2=x^3(x+1)^5{M_1}^2, C_3=x^4(x+1)^4{M_1}^2,\\
C_4=x^6(x+1)^6{M_1}^2,
C_5=x^4(x+1)^5{M_1}^3,  C_6=x^7(x+1)^8{M_5},\\
C_7=x^7(x+1)^9{M_5}^2,
C_8 =x^8(x+1)^8M_4M_5, C_9 =x^8(x+1)^9M_4{M_5}^2,\\
C_{10}=x^7(x+1)^{10}{M_1}^2M_5,
C_{11}=x^7(x+1)^{13}{M_2}^2{M_3}^2, \\
C_{12}=x^9(x+1)^{9}{M_4}^2{M_5}^2, C_{13}=x^{14}(x+1)^{14}{M_2}^2{M_3}^2,\\
D_{1}=x^4(x+1)^{5}{M_1}^4S_1, D_{2}=x^4(x+1)^{5}{M_1}^5{S_1}^2.\\
\\
\text{The set ${\cal{U}}:=\{M_1, \ldots, M_5\} \subset {\mathcal{M}}$ and $S_1 \not\in {\mathcal{M}}$.}
\end{array}$$
\begin{theorem} \label{caseomega3}
Let $A \in \F_2[x]$ be b.u.p. such that $\omega(A)=3$. Then \\
$A, \overline{A} \in \{C_j: j \leq 7\}$.
\end{theorem}
\begin{theorem} \label{caseomega4}
Let $A \in \F_2[x]$ be b.u.p. such that $\omega(A) =4$. Then \\
$A, \overline{A} \in \{C_j: 8 \leq j \leq 13\} \cup \{D_1, D_2\}$.
\end{theorem}

\section{Preliminaries}\label{preliminaire}
We need the following results. Some of them are obvious or (well) known, so we omit
their proofs.
\begin{lemma} \label{gcdunitary}
Let $T$ be an irreducible polynomial over $\F_2$ and $k,l\in \N^*$.
Then,  $\gcd_u(T^k,T^l) = 1 \ ($resp. $T^k)$ if $k \not= l \ ($resp. $k=l)$.\\
In particular, $\gcd_u(T^k,T^{2n-k}) = 1$ for $k \not=n$,  $\gcd_u(T^k,T^{2n+1-k}) = 1$ for
any $0\leq k \leq 2n+1$.
\end{lemma}
\begin{lemma} \label{aboutsigmastar2}
Let $T \in \F_2[x]$ be irreducible. Then\\
i) $\sigma^{**}(T^{2n}) = (1+T)\sigma(T^n) \sigma(T^{n-1}), \
\sigma^{**}(T^{2n+1}) = \sigma(T^{2n+1})$.\\
ii) For any $c \in \N$, $T$ does not divide $\sigma^{**}(T^{c})$.
\end{lemma}
\begin{proof}
i): $\sigma^{**}(T^{2n}) = 1+T+\cdots+T^{n-1} + T^{n+1}+\cdots+T^{2n} =
(1+T^{n+1})\sigma(T^{n-1}) = (1+T)\sigma(T^n) \sigma(T^{n-1})$,
$\sigma^{**}(T^{2n+1}) = 1+T+\cdots+T^{2n+1}$.\\
ii) follows from i).
\end{proof}
\begin{corollary} \label{expressigmastar} Let $T \in \F_2[x]$ be irreducible. Then\\
i) If $a \in \{4r,4r+2\}$, where $2r-1$ or $2r+1$ is of the form $2^{\alpha}u-1$, $u$ odd, then
$\sigma^{**}(T^a)=(1+T)^{2^{\alpha}} \cdot \sigma(T^{2r}) \cdot (\sigma(T^{u-1}))^{2^{\alpha}},
\ \gcd(\sigma(T^{2r}), \sigma(T^{u-1}))=1.$\\
ii) If $a=2^{\alpha}u-1$ is odd, with $u$ odd, then
$\sigma^{**}(T^a)=(1+T)^{2^{\alpha}-1} \cdot (\sigma(T^{u-1}))^{2^{\alpha}}$.
\end{corollary}
\begin{corollary} \label{splitcriteria}
i) The polynomial $\sigma^{**}(x^{a})$ splits over $\F_2$ if and only if $a=2$ or $a=2^{\alpha}-1$, for some $\alpha \in \N^*$.\\
ii) Let $T \in \F_2[x]$ be odd and irreducible. Then, $\sigma^{**}(T^{c})$ splits over $\F_2$ if and only if
$(T$ is Mersenne, $c=2$ or $c=2^{\gamma}-1$ for some $\gamma \in \N^*).$\\
iii) Let $T \in \F_2[x]$ be odd and irreducible such that $\sigma^{**}(T^{c})$ does not split over $\F_2$, then $\sigma^{**}(T^{c})$ is divisible by $\sigma(T^{2m})$, for some $m \geq 1$.
\end{corollary}

\begin{lemma} \label{nombreminimal}
If $A$ is a nonconstant b.u.p.
polynomial over $\F_2$, then $x(x+1)$ divides $A$ so that $\omega(A) \geq 2$.
\end{lemma}

\begin{lemma} \label{multiplicativity}
If $A = A_1A_2$ is b.u.p. over $\F_2$ and if
$\gcd(A_1,A_2) =~1$, then $A_1$ is b.u.p. if and only if
$A_2$ is b.u.p.
\end{lemma}
\begin{lemma} \label{translation}
If $A$ is b.u.p. over $\F_2$, then the polynomial $\overline{A}$
is also b.u.p. over~$\F_2$.
\end{lemma}
Lemma \ref{complete2} below gives some useful results from Canaday's paper
(\cite{Canaday}, Lemmas 4, 5, 6, Theorem 8 and Corollary on page 728).
\begin{lemma} \label{complete2}
Let $P, Q \in \F_2[x]$ be such that $P$ is irreducible and let $n,m \in \N$.\\
\emph{i)} If $\sigma(P^{2n}) = Q^m$, then $m
\in \{0,1\}$.\\
\emph{ii)} If $\sigma(P^{2n}) = Q^m T$, with $m > 1$ and $T \in
\F_2[x]$ is nonconstant, then ${\rm{deg}}(P)
> {\rm{deg}}(Q)$.\\
\emph{iii)} If $P$ is a Mersenne prime and if $P=P^*$, then $P \in \{M_1, M_4\}$.\\
\emph{iv)} If $\sigma(x^{2n}) = PQ$ and $P = \sigma((x+1)^{2m})$, then $2n=8$, $2m=2$, $P= M_1$ and
$Q = P(x^3) = 1+x^3+x^6$.\\
\emph{v)} If any irreducible factor of $\sigma(x^{2n})$ is a Mersenne prime, then $2n \leq 6$.\\
\emph{vi)} If $\sigma(x^{2n})$ is a Mersenne prime, then $2n \in \{2,4\}$.\\
\emph{vii)} If $\sigma(x^{n}) = \sigma((x+1)^{n})$,
then $n= 2^h-2$, for some $h \in \N^*$.
\end{lemma}
\begin{lemma} \label{squarefreeMers} {\rm{[see \cite{Gall-Rahav13}, Lemma 2.6]}} Let $m \in \N^*$ and $T$ be a Mersenne prime. Then, $\sigma(x^{2m})$, $\sigma((x+1)^{2m})$ and $\sigma(M^{2m})$ are all odd and squarefree.
\end{lemma}
Since we suppose that $A = x^a(x+1)^bP^cQ^d$, where $P$ and $Q$ are odd and irreducible, the following equalities (obtained from Corollary \ref{expressigmastar}) are useful:
\begin{equation} \label{lessigmastaralla}
\left\{\begin{array}{l}
\sigma^{**}(x^a)=(1+x)^{2^{\alpha}} \cdot \sigma(x^{2r}) \cdot (\sigma(x^{u-1}))^{2^{\alpha}},
\text{ with $\gcd(\sigma(x^{2r}), \sigma(x^{u-1}))=1$},\\
\text{if $a = 4r, 2r-1=2^{\alpha}u-1$, (resp. $a=4r+2, 2r+1=2^{\alpha}u-1$), $u$ odd}, r, \alpha \geq 1,\\
\sigma^{**}(x^a)=(1+x)^{2^{\alpha}-1} \cdot (\sigma(x^{u-1}))^{2^{\alpha}},
\text{if $a=2^{\alpha}u-1$ is odd, $u$ odd},
r, \alpha \geq 1.\\
\text{Set $\varepsilon_1 := \min(1,u-1) \in \{0,1\}$}.
\end{array}
\right.
\end{equation}
\begin{equation} \label{lessigmastarallb}
\left\{\begin{array}{l}
\sigma^{**}((x+1)^b)=x^{2^{\beta}} \cdot \sigma((x+1)^{2s}) \cdot
(\sigma((x+1)^{v-1}))^{2^{\beta}},\\
\text{if $b = 4s, 2s-1=2^{\beta}v-1$, (resp. $b=4s+2, 2s+1=2^{\beta}v-1$), $v$ odd}, s, \beta \geq 1,\\
\sigma^{**}((x+1)^b)=x^{2^{\beta}-1} \cdot
(\sigma((x+1)^{v-1}))^{2^{\beta}},
\text{if $b=2^{\beta}v-1$ is odd, with $v$ odd}, s, \beta \geq 1.\\
\text{Set $\varepsilon_2 := \min(1,v-1) \in \{0,1\}$}.
\end{array}
\right.
\end{equation}~\\
\begin{equation} \label{lessigmastar-add3}
\left\{\begin{array}{l}
\sigma^{**}(P^c)= (1+P)^{2^{\gamma}} \cdot \sigma(P^{2t}) \cdot (\sigma(P^{w-1}))^{2^{\gamma}},
\text{ with $\gcd(\sigma(P^{2t}), \sigma(P^{w-1}))=1$},\\
\text{if $c \in \{4t,4t+2\}$, where $2t-1$ or $2t+1$ is of the form $2^{\gamma}w-1$, $w$ odd},\\
\\
\sigma^{**}(P^c)=(1+P)^{2^{\gamma}-1} \cdot (\sigma(P^{w-1}))^{2^{\gamma}},
\text{ if $c=2^{\gamma}w-1$ is odd, $w$ odd},\\
\\
\sigma^{**}(Q^d)=(1+Q)^d =  x^{u_2d}(x+1)^{v_2d}P^{w_2d}\\
\text{if $d \in \{2, 2^{\gamma}-1: \gamma \geq 1\}$
and $Q = 1+x^{u_2}(x+1)^{v_2}P^{w_2}$}, 
t, \gamma, u_2, v_2, w_2 \geq 1.
\end{array}
\right.
\end{equation}
\section{Proof of Theorem \ref{caseomega3}} \label{caseomeg=3}
We set $A=x^a(x+1)^bP^c$, with $a,b,c \in \N^*$ and $P$ odd irreducible.\\
We suppose that
$A$ is b.u.p.:
$$\sigma^{**}(x^a) \cdot \sigma^{**}((x+1)^b) \cdot \sigma^{**}(P^c) =\sigma^{**}(A)
= A=x^a(x+1)^bP^c.$$
We show that $P$ is a Mersenne prime and we apply Theorem 1.1 in \cite{Gall-Rahav15}.
\begin{lemma} \label{aoubsup3}
The polynomial $\sigma^{**}(x^a (x+1)^b)$ does not split, so that $(a\geq 3$ or $b\geq 3)$ and
$(a \not= 2^n-1$ or $b \not= 2^m-1$ for any $n,m \geq 1)$.
\end{lemma}
\begin{proof}
If $\sigma^{**}(x^a (x+1)^b)$ splits, then $\sigma^{**}(x^a (x+1)^b) = x^b(x+1)^a$. Thus, $a=b$ and $\sigma^{**}(P^c) = P^c$. It contradicts Lemma \ref{aboutsigmastar2}-ii).\\
If $a, b\leq 2$ or ($a = 2^n-1$, $b = 2^m-1$ for some $n,m \geq 1$), then $\sigma^{**}(x^a)$ and $\sigma^{**}((x+1)^b)$ split.
\end{proof}

\begin{corollary}
The polynomial $P$ is a Mersenne prime.
\end{corollary}
\begin{proof}
The polynomial $P$ does not divide $\sigma^{**}(P^{c})$ (Lemma \ref{aboutsigmastar2}). So, $\sigma^{**}(P^{c})$ splits and $P \in {\cal{M}}$, by Corollary \ref{splitcriteria}.
\end{proof}

\section{Proof of Theorem \ref{caseomega4}} \label{proofomega4}
In this section, we set $A=x^a(x+1)^bP^cQ^d$, with $a,b,c,d \in \N^*$, $P, Q$ odd irreducible, and
$\deg(P) \leq \deg(Q)$.
We suppose that $A$ is b.u.p.:
$$\sigma^{**}(x^a) \cdot \sigma^{**}((x+1)^b) \cdot \sigma^{**}(P^c)
\cdot \sigma^{**}(Q^d) =\sigma^{**}(A)= A=x^a(x+1)^bP^cQ^d.$$
We prove in Lemma \ref{valuesofd}, that $P \in {\mathcal{M}}$ and $(Q \in {\mathcal{M}}$ or it is of the form $1+x^{u_2}(x+1)^{v_2}P^{w_2}$, where $u_2, v_2, w_2 \geq 1$).
\begin{lemma} \label{valuesofd}
i) The polynomial $P$ is a Mersenne prime.\\
ii) The integer $d$ equals $2$ or it is of the form $d=2^{\delta}-1$, with $\delta \in \N^*$.\\
iii) The polynomial $Q$ is of the form $1+x^{u_2}(x+1)^{v_2}P^{w_2}$, where $w_2 \in \{0,1\}$.\\
iv) One has: $a, b \geq 3$ and $d \leq \min(a,b)$.\\
v) If $\sigma^{**}(P^c)$ does not split, then $Q$ is its unique odd divisor.
\end{lemma}
\begin{proof}
i): We remark that $1+P$ divides $\sigma^{**}(P^c)$. If $1+P$ does not split over $\F_2$,
then $Q$ is an odd irreducible divisor of $1+P$ and we get the contradiction:
$\deg(Q) < \deg(P) \leq \deg(Q)$.\\
ii): If $d$ is even and if $d \geq 4$, then $d$ is of the form $4r$ or $4r+2$. Thus, the odd polynomial
$\sigma(Q^{2r})$
divides $\sigma^{**}(A) = A$,
so we must have $P=\sigma(Q^{2r})$, which contradicts the fact: $\deg(P) \leq \deg(Q)$.\\
If $d=2^{\delta}w-1$ is odd (with $w$ odd) and if $w \geq 3$, then
$P=\sigma(Q^{w-1})$ and $\deg(P) > \deg(Q)$, which is impossible.\\
iii): From ii), $\sigma^{**}(Q^d)= (1+Q)^d$ so that $(1+Q)^d$ divides $A$. We may put: $1+Q =x^{u_2}(x+1)^{v_2}P^{w_2}$, for some $u_2,v_2,w_2 \in \N$, $u_2, v_2 \geq 1$.\\
iv): $a, b \geq 3$ because $1+x$ divide $\sigma^{**}(x^a)$, $x$ divides $\sigma^{**}((x+1)^b)$ and
$x(x+1)$ divides both
$\sigma^{**}(P^c)$ and $\sigma^{**}(Q^d)$.\\
From the proof of iii),
$x^{du_2}$ and $(x+1)^{dv_2}$ both divide $A$. Thus, $d \leq \min(a,b)$.\\
v) is immediate.
\end{proof}

\subsection{Case where $Q \in {\mathcal{M}}$} \label{caseQmers}
Here, $P$ and $Q$ are both Mersenne, so we apply Theorem 1.1 in \cite{Gall-Rahav15}.

\subsection{Case where $Q \not\in {\mathcal{M}}$} \label{caseQnonmers}
We prove Proposition \ref{QnonMersenne}.
\begin{proposition} \label{QnonMersenne}
If $A$ is b.u.p., where $P \in {\mathcal{M}}$ but $Q \not\in {\mathcal{M}}$, then $A, \overline{A} \in \{D_1, D_2\}$.
\end{proposition}

\subsubsection{Useful facts} \label{usefulfacts}
As in Lemma \ref{aoubsup3}, one has: $a\geq 3$ or $b\geq 3$.
Lemma \ref{valuesofd} allows to write: $P=1+x^{u_1}(x+1)^{v_1}$ and
$Q=1+x^{u_2}(x+1)^{v_2}P^{w_2}$, with $u_i,v_j,w_2 \geq 1$.
We obtain Corollaries \ref{synthese4}, \ref{synthese5} and \ref{synthese3}. Only, the last of them gives b.u.p. polynomials, namely $D_{1}, D_{2}$, $\overline{D}_1$ and $\overline{D}_2$ (see Section \ref{compute}).\\
For any $g \geq 1$, $PQ$ is not of the form $\sigma(P^{2g})$, because $P$ does not divide $\sigma(P^{2g})$. We shall see that it suffices to consider three cases (replace $A$ by $\overline{A}$, if necessary):
$\text{$PQ=\sigma(x^{2m})$, $Q=\sigma(x^{2m})$, $Q=\sigma(P^{2m})$, for some $m\geq 1$.}$
\begin{lemma} \label {lesodddivisors}
i) Let $n \geq 1$ be such that $\sigma(x^{2n})$ $($resp. $\sigma((x+1)^{2n})$, $\sigma(P^{2n}))$ divides $\sigma^{**}(A)$, then $\sigma(x^{2n}) \in \{P,Q,PQ\}$ $($resp. $\sigma((x+1)^{2n}) \in \{P,Q,PQ\}$, $\sigma(P^{2n})=Q)$.\\
ii) For any $n \geq 1$, $\sigma(Q^{2n})$ does not divide $\sigma^{**}(A)$.
\end{lemma}
\begin{proof} Recall that we suppose: $\sigma^{**}(A)=A$.\\
i): $\sigma(x^{2n})$, $\sigma((x+1)^{2n})$ and $\sigma(P^{2n})$ are all odd and
squarefree (Lemma \ref{squarefreeMers}). Hence, they belong to $\{P,Q,PQ\}$ whenever they divide $\sigma^{**}(A)$, with $\sigma(P^{2n}) \not\in  \{P,PQ\}$.\\
ii): If $\sigma(Q^{2n}) \mid \sigma^{**}(A)$, then $P^m=\sigma(Q^{2n})$, with $m=1$, by
Lemma \ref{complete2}-i). So, we get the
contradiction: $\deg(Q) \geq \deg(P) = 2n \deg(Q) > \deg(Q)$.
\end{proof}

\begin{lemma} [\cite{Canaday}, Lemma 4, page 726]  \label{Canadaypage726}~\\
The polynomial $1+x(x+1)^{2^{\nu}-1}$ is irreducible if and only if $\nu \in \{1,2\}$.
\end{lemma}
\begin{lemma} \label{sigmastar2P}
If $\sigma(P^{2n})$ divides $A$ for some $n \geq 1$, then $2n=2^{\gamma}$, $2n-1 \leq \min(a,b)$.
\end{lemma}
\begin{proof}
Since $\sigma(P^{2n})$ is odd and square-free, $Q$ must divide it.
So $Q=\sigma(P^{2n})$.
Put: $2n = 2^{\gamma} h$, with $h$ odd.\\
We get:
$\displaystyle{1+P+\cdots +P^{2n-1} = \frac{1+\sigma(P^{2n})}{P} = \frac{1+Q}{P} = x^{u_2}(x+1)^{v_2}P^{w_2 -1}}$. Thus, $w_2=1$ and
$(1+P)^{2^{\gamma}-1} (1+P+\cdots+P^{h-1})^{2^{\gamma}} = 1+P+\cdots + P^{2n-1} = x^{u_2}(x+1)^{v_2}$.
Hence, $h=1$, $2n-1 \leq (2^{\gamma}-1)u_1 = u_2 \leq a$ and $2n-1 \leq (2^{\gamma}-1)v_1 = v_2 \leq b$.
\end{proof}
\begin{lemma} \label{lesQirreductibles}
i) Let $P=M_4$ and $Q=1+x^5(x+1)^{2^{\nu}-1} P^{2^{\nu}-1}$,
with $\nu \geq 1$. Then, $Q$
is irreducible if and only if $\nu=2$.\\
ii) Let $P \in \{M_1, M_4\}$ and $Q=1+x(x+1)^{2^{\nu}-1}P^{2^{\nu}}$, with $\nu \leq 10$.
Then, $Q$ is
irreducible if and only if
$(\nu=2$, $P=M_1)$ or $(\nu=1$, $P=M_4)$.\\
iii) Let $P \in \{M_1, M_4\}$ and $Q=1+P(1+P)^{2^{\nu}-1}$. Then, $Q$ is irreducible
if and only if $P=M_1$ and $\nu \in \{1,2\}$.
\end{lemma}
\begin{proof}
i): One has $Q=1+x^5(x+1)^{2^{\nu}-1} P^{2^{\nu}-1} = 1+x^5(x^5+1)^{2^{\nu}-1}$. The irreducibility of
$Q$  implies that $1+x(x+1)^{2^{\nu}-1}$ is irreducible. So, $\nu \in \{1,2\}$
by Lemma \ref{Canadaypage726}.\\
If $\nu = 1$, then $Q=1+x^5+x^{10} = (x^4+x+1)M_1M_5$ is reducible.\\
If $\nu = 2$, then $Q=1+x^5+x^{10}+x^{15}+x^{20}$ which is irreducible.\\
ii): by direct (Maple) computations.\\
iii): The polynomial $U=1+x(x+1)^{2^{\nu}-1}$ must be irreducible, so $\nu \in \{1,2\}$
by Lemma \ref{Canadaypage726}. Thus, $U \in \{M_1, M_4\}$.\\
If $P=U=M_1$, then $Q=1+x+x^4=1+x(x+1)P$ is irreducible.\\
If $P=M_1$ and $U=M_4$, then $Q=1+x^3(x+1)^3P$ is irreducible.\\
If $P=M_4$ and $U=M_1$, then $Q=1+x(x+1)^3P = (x^6+x^5+x^4+x^2+1)M_1$ is reducible.\\
If $P=U=M_4$, then
$Q=1+x^3(x+1)^9P = (x^{12}+x^9+x^8+x^7+x^6+x^4+x^2+x+1)(1+x+x^4)$ is reducible.
\end{proof}

\begin{lemma} \label{factorifPQ=sigmx2m}
 If $PQ=\sigma(x^{2n})$, then $(2n=8, P=M_1, Q=1+x^3+x^6)$ or
$(2n=24, P=M_4, Q= 1+x^5(x^5+1)^3)$. Moreover, $Q, \overline{Q} \not\in \{\sigma(x^{2g}), \sigma(P^{2g}): g \geq 1\}$ and
$PQ \not\in \{\sigma(x^{2g}), \sigma((x+1)^{2g}) : g \geq 1\}$.
\end{lemma}
\begin{proof}
Since $PQ=\sigma(x^{2n})$, we get $P=P^*$ or $P=Q^*$. But, here, $\deg(P) < \deg(Q)$.
So, $P=P^*$ and $Q=Q^*$.
Since $P$ is a Mersenne prime and $P=P^*$, one has $P=M_1$ or $P=M_4$.
If $P=M_1$,
then by Lemma \ref{complete2}-iv), $Q= 1+x^3(x+1)P=1+x^3+x^6$.
If $P=M_4$, then direct computations give $Q=1+x^5(x+1)^{2^{\nu} -1}P^{2^{\nu} -1}$. Since $Q$ is irreducible, we get from Lemma \ref{lesQirreductibles}-i), $\nu = 2$ and $Q = 1+x^5(x^5+1)^3$. Thus,
$Q \not\in \{\sigma(x^6), \sigma((x+1)^6)\}$ (resp. $Q \not\in \{\sigma(x^{20}), \sigma((x+1)^{20})\}$ if $P =M_1$ (resp. if $P =M_4$).
We also remark that $\displaystyle{\frac{\deg(Q)}{\deg(P)} \in \{3,5\}}$. So, $Q, \overline{Q} \not\in \{\sigma(P^{2g}): g\geq 1\}$.
\end{proof}
\begin{lemma} \label{factorifQ=sigmx2m}
If $Q=\sigma(x^{2n})$ with $n \geq 1$, then for some $\nu \geq 1$, $Q = 1+x(x+1)^{2^{\nu}-1} {M_1}^{2^{\nu}}$ or $Q= 1+x(x+1)^{2^{\nu}-1} {M_4}^{2^{\nu}}$. Moreover, $Q, \overline{Q} \not\in \{\sigma(P^{2g}): g \geq 1\}$ and
$PQ \not\in \{\sigma(x^{2g}), \sigma((x+1)^{2g}) : g \geq 1\}$.
\end{lemma}
\begin{proof}
By direct computations, one has, for some $\nu \geq 1$: $2n=2^{\nu}t$, $t \in \{3,5\}$, $P=\sigma(x^{t-1})$
and $Q=1+x(x+1)^{2^{\nu}-1}P^{2^{\nu}}$. Hence, $P^{2^{\nu}} \| 1+Q$. \\
If $PQ$ is of the form $\sigma(x^{2g})$, then $P \| 1+Q$ or $P^3 \| 1+Q$ (Lemma \ref{factorifPQ=sigmx2m}), which is impossible.\\
Since $Q = \sigma(x^{2m})$, Lemma \ref{lessigmadistincts}-i) implies that $Q \not\in \{\sigma(P^{2m}), \sigma({\overline{P}}^{\ 2m})\}$.
\end{proof}
\begin{lemma} \label{factorifQ=sigmP2m}
If $Q=\sigma(P^{2n})$, then $2n \leq 4$, $P =M_1$, so that $Q \in \{1+x(x+1)M_1, 1+x^3(x+1)^3M_1\}$. Moreover,
$Q, PQ \not\in \{\sigma(x^{2g}), \sigma((x+1)^{2g}): g \geq 1\}$.
\end{lemma}
\begin{proof}
By direct computations, one has: $2n=2^{\nu}, Q=1+P(1+P)^{2^{\nu}-1}$, for some $\nu \geq 1$. Since $Q$ is irreducible, we get $\nu \in\{1,2\}$ and $P=M_1$. Again, by direct computations, $Q, PQ \not\in \{\sigma(x^{2g}), \sigma((x+1)^{2g}): g \geq 1\}$.
\end{proof}

\begin{lemma} \label{lessigmadistincts}
i) For any $m,n \in \N^*$, $\sigma(P^{2m}) \not\in \{\sigma(x^{2n}), \sigma((x+1)^{2n})\}$.\\
ii) If $\sigma(x^{2n}) = \sigma((x+1)^{2n})$, then  $\sigma(x^{2n}) \not\in \{Q, PQ\}$.
\end{lemma}
\begin{proof}
i): Put $2n-1 = 2^{\alpha}u-1$ and $2m-1 = 2^{\beta}v-1$, with $\alpha, \beta \geq 1$. \\
If $\sigma(P^{2m}) = \sigma(x^{2n})$, then $P(1+P+\cdots +P^{2m-1}) = x(1+x+\cdots+x^{2n-1})$.
Thus, $P(P+1)^{2^{\beta} - 1} (1+P+\cdots + P^{v-1})^{2^{\beta}} = x(x+1)^{2^{\alpha} - 1} (1+x+\cdots + x^{u-1})^{2^{\alpha}}.$ Hence, $u \geq 3$ and $2^{\alpha} = 1$, which is impossible.\\
ii): One has $2n = 2^{h} - 2$, for some $h \geq 1$ (Lemma \ref{complete2}-vii)). If $Q = \sigma(x^{2n})$, then by Lemma \ref{factorifQ=sigmx2m}, $2^{h} - 2=2n=2^{\nu} t$, with $t\in \{3,5\}$.
Therefore, $\nu =1$, $t= 2^{h-1} - 1$, $h=3=t$, $2n=6$ and $Q = M_2M_3$ is reducible.\\
If $PQ = \sigma(x^{2n})$, then by Lemma \ref{factorifPQ=sigmx2m}, one has: ($2n=8$, $P=M_1$ and $Q=1+x^3+x^6$) or
($2n=5\cdot 2^{\nu} +4$, $P=M_4$ and
$Q=1+x^5(x+1)^{2^{\nu}-1}P^{2^{\nu}-1}$). Thus, $2^{h} - 2= 2n=5\cdot 2^{\nu} +4$, $\nu=1$, $h=4$ and
$Q=1+x^5(x+1)P=(x^4+x+1)M_1M_5$ is reducible.
\end{proof}
Without loss of generality, by Lemmas  \ref{factorifPQ=sigmx2m}, \ref{factorifQ=sigmx2m} and \ref{factorifQ=sigmP2m}, it suffices to consider the following three cases:\\
$$PQ=\sigma(x^{2m}), \ Q=\sigma(x^{2m}), \ Q=\sigma(P^{2m}), \text{for some $m\geq 1$}.$$
In each case, we distinguish: $\text{($a$, $b$ both even), ($a$ even, $b$ odd), ($a$, $b$ both odd).}$
We shall compare $a$, $b$, $c$ or $d$ with all possible values of the exponents of $x$, $x+1$,  of $P$ or of $Q$,
in $\sigma^{**}(A)$.

According to Corollary \ref{expressigmastar} and Lemma \ref{valuesofd}, we get Lemma \ref{Petsigmastar2P} from  Relations in (\ref{lessigmastaralla}), in (\ref{lessigmastarallb}) and in (\ref{lessigmastar-add3}).
\begin{lemma} \label{Petsigmastar2P}~\\
i) The polynomial $P$ does not divide $\sigma^{**}(P^c)$, but it may divide $\sigma^{**}(Q^d)$.\\
ii) One has: $u_2d \leq a,\ v_2d \leq b,\ w_2d \leq c$, so that $d \leq \min(a,b,c)$.
\end{lemma}

\subsubsection{Case where $PQ = \sigma(x^{2m})$, for some $m \geq 1$} \label{casePQ=sigmx2r}
We get, from Lemma \ref{factorifPQ=sigmx2m}, $Q, \overline{Q} \not\in \{\sigma(x^{2g}), \sigma(P^{2g}): g\geq 1\}$,
$(2m=8, P=M_1$ and $Q=1+x^3+x^6 = 1+x^3(x+1)P)$ or $(2m=24, P=M_4$ and $Q=1+x^5(x^5+1)^3 = 1+x^5(x+1)^3P^3)$.\\
We refer to Relations in (\ref{lessigmastaralla}), in (\ref{lessigmastarallb}) and in (\ref{lessigmastar-add3}).
\begin{lemma} \label{valeursdec}
On has: $c=2$ or $c=2^{\gamma}-1$, $c \leq \min(a,b)$ and $d=1$.
\end{lemma}
\begin{proof}
Since $Q \not= \sigma(P^{2g})$ for any $g$, $\sigma^{**}(P^c)$ must split, so
$c=2$ or $c=2^{\gamma}-1$. In this case, $\sigma^{**}(P^c) =(1+P)^c$, where $P$ is a Mersenne prime. So, $x^c$ and
$(x+1)^c$ both divide $\sigma^{**}(A)=A$. Hence, $c \leq \min(a,b)$. Finally, $Q \| \sigma^{**}(A)$ because
$Q, \overline{Q} \not\in \{\sigma(x^{2g}), \sigma(P^{2g}): g\geq 1\}$. Thus, $d=1$.
\end{proof}
\begin{lemma}
The integers $a$ and $b$ are not both odd.
\end{lemma}
\begin{proof}
If $a$ and $b$ are both odd, then $PQ=\sigma(x^{u-1})$,
$\sigma((x+1)^{v-1}) \in \{1,P\}$, $d=2^{\alpha}$,
$c =w_2d+2^{\alpha}+ \varepsilon_2 2^{\beta}$. It follows that $c$ is even and $c\geq 4$, which
contradicts Lemma \ref{valeursdec}.
\end{proof}
\begin{lemma}
If $a$ and $b$ are both even, then $a=16$, $b \in \{4,6\}$,
$c \leq 3$, $P=M_1$ and $Q = 1+x^3(x^3+1)$.
\end{lemma}
\begin{proof}
Lemma \ref{valuesofd}-iv) implies that $a,b\geq 4$. Moreover,  $PQ \in \{\sigma(x^{2r}),\sigma(x^{u-1})\}$. If $PQ=\sigma(x^{2r})$, then $P=\sigma((x+1)^{2s})$, $u=v=1$ because
$\gcd(\sigma(x^{2r}),\sigma(x^{u-1})) = 1 = \gcd(\sigma((x+1)^{2s}),\sigma((x+1)^{v-1}))$.
Therefore, $2r=8$, $a \not= 4r+2$, $2s=2$, $a=16$, $b\in \{4,6\}$.
Furthermore, $c\leq b\leq 6$, so that $c \in \{1,2,3\}$.\\
If $PQ = \sigma(x^{u-1})$, then $\sigma(x^{2r}) = P$ (by Lemma \ref{lesodddivisors}), which is impossible since
$\gcd(\sigma(x^{2r}),\sigma(x^{u-1})) = 1$.
\end{proof}

\begin{lemma}
If $a$ is even and $b$ odd, then $a=16$, $b\in \{1,3,7\}$, $c=2$, $P=M_1$ and $Q = 1+x^3(x^3+1)$.
\end{lemma}
\begin{proof}
As above, $a$ even implies that $a=4r=16$ and $P=M_1$.
One has:
$\sigma((x+1)^{v-1}) \in \{1,P\}$. So, $v\in \{1,3\}$,
$c=1+w_2d+\varepsilon_2 \ 2^{\beta}$,
$w_2=1=d$. Thus, $c=2$, $v=1$, $2^{\beta}-1+3+2 \leq a= 16$,
$\beta \leq 3$ and $b\in \{1,3,7\}$.
\end{proof}
\begin{corollary} \label{synthese4}
If $A$ is b.u.p., with $PQ$ of the form $\sigma(x^{2m})$, then $P=M_1$, $Q=1+x^3(x^3+1),$
$a,b \in \{1,3,4,6,7,16\}$, $c \leq 3$ and $d=1$.
\end{corollary}

\subsubsection{Case where $Q = \sigma(x^{2m})$, for some $m \geq 1$} \label{caseQ=sigmx2r}
One has (Lemma \ref{factorifQ=sigmx2m}): $Q, \overline{Q} \not\in \{\sigma(P^{2g}): g\geq 1\}$,
$PQ \not\in \{\sigma(x^{2g}), \sigma((x+1)^{2g}): g\geq 1\}$, $2m \geq 10$, $P \in \{M_1,M_4\}$ and $Q=1+x(x+1)^{2^{\nu}-1} P^{2^{\nu}}$, for some $\nu \in \N^*$. So, $u_1 =u_2= 1$, $v_1 \in \{1,3\}$, $v_2 = 2^{\nu}-1$ and $w_2 = 2^{\nu}$.
Moreover, $Q \not= \sigma((x+1)^{2m})$ (Lemma \ref{lessigmadistincts}). \\
We consider Relations in (\ref{lessigmastaralla}), in (\ref{lessigmastarallb}) and in (\ref{lessigmastar-add3}).
\begin{lemma} \label{cegal2ou2alpha-1}
One has: $(c=2$ or $c=2^{\gamma}-1)$ and $d \leq 3$.
\end{lemma}
\begin{proof}
If $\sigma^{**}(P^c)$ does not split, then $Q$ is the unique odd irreducible divisor of
$\sigma^{**}(P^c)$.
It contradicts the fact that $Q$ is not of the form $\sigma(P^{2g})$.
So, $\sigma^{**}(P^c)$ splits and $(c=2$ or $c=2^{\gamma}-1)$.
The exponent of $Q$ in $\sigma^{**}(A)$
lies in $\{1,2,2^{\alpha}, 2^{\beta}, 1+2^{\alpha}, 1+2^{\beta}, 2^{\alpha} + 2^{\beta}\}$. So, by Lemma \ref{valuesofd}-ii), $d \leq 3$.
\end{proof}
\begin{lemma}
The integers $a$ and $b$ are not both odd.
\end{lemma}
\begin{proof}
If $a$ and $b$ are both odd, then $Q=\sigma(x^{u-1})$, $Q \not= \sigma((x+1)^{v-1})$ (by Lemma \ref{lessigmadistincts}-ii)) and $\sigma((x+1)^{v-1}) \in \{1,P\}$. Thus, $v \in \{1,3,5\}$,
$2^{\alpha} = d \leq 3$, $\alpha =1$, $d = 2$, $c = 2 \cdot 2^{\nu}+\varepsilon_2 2^{\beta}$. 
So, $c$ is even and $c \geq 4$.
It contradicts Lemma \ref{cegal2ou2alpha-1}.
\end{proof}
\begin{lemma}
If $a$ and $b$ are even, then $\nu \leq 2$, $20 \leq a \leq 26$, $b \leq 10$, $d=1$,
$c \in \{1,2,3,7\}$,
 and
$\text{$(P,Q) \in \{(M_1, 1+x(x+1)^{3}P^{4}), (M_4, 1+x(x+1)P^{2})\}$}.$
\end{lemma}
\begin{proof}
One has: $Q \in \{\sigma(x^{2r}), \sigma(x^{u-1})\}$. \\
- If $Q = \sigma(x^{2r})$, then $Q \not= \sigma((x+1)^{2s})$ (by Lemma \ref{lessigmadistincts}-ii)), $Q$ does not divide  $\sigma(x^{u-1})$ since $\gcd(\sigma(x^{2r}),\sigma(x^{u-1})) = 1$. So, $Q \| \sigma^{**}(A)$. Therefore, $d=1$, $P=\sigma((x+1)^{2s})$, $\sigma(x^{u-1}) \in \{1,P\}$,
$u \in \{1,3,5\}$, $v=1$, $2s\leq 4$, $b\leq 10$,
$c=2^{\nu} + \varepsilon_1 2^{\alpha} + 1 \geq 3$. 
Since $2^{\alpha} + c\leq b \leq 10$, we get: $c \in \{1,2,3,7\}$, $\alpha \leq 2$, $\nu \leq 2$.\\
Here, $Q=1+x(x+1)^{2^{\nu}-1}P^{2^{\nu}}$, with
$P \in \{M_1, M_4\}$ and $\nu \leq 2$. By Lemma \ref{lesQirreductibles}-ii), one has: $(P=M_1$,
$\nu=2$ and $2r=12)$ or $(P=M_4$, $\nu=1$ and $2r=10)$. So, $20 \leq a \leq 26$.\\
- If $Q = \sigma(x^{u-1})$, then $2^{\alpha} = d \leq 3$ and $P = \sigma(x^{2r}) = \sigma((x+1)^{2s})$. Thus, $d=2$, $2r=2s=2$, $a,b \in \{4,6\}$, $c = 2 + w_2d = 2+2w_2 \geq 4$. It contradicts Lemma \ref{cegal2ou2alpha-1}.
\end{proof}

\begin{lemma}
The case where $a$ is even and $b$ odd does not happen.
\end{lemma}
\begin{proof}
If $a$ is even and $b$ odd, then  $Q \in \{\sigma(x^{2r}), \sigma(x^{u-1})\}$. \\
- If $Q=\sigma(x^{2r})$, then $d=1$, $\sigma(x^{u-1}), \sigma((x+1)^{v-1}) \in \{1,P\}$,
$u,v \in \{1,3,5\}$, $w_2d = 2^{\nu}$, $c=2^{\nu} + \varepsilon_1 2^{\alpha} + \varepsilon_2 2^{\beta}$ is even.\\
Therefore, $c=2$, $\nu=1$, $\varepsilon_1=\varepsilon_2=0$ and $u=v=1$.\\
By Lemma \ref{lesQirreductibles}-ii), since $\nu=1$, one has: $P=M_4$ and thus $v_1=3, v_2 = 1, w_2 = 2$,
$2r=\deg(Q)=2^{\nu}(1+\deg(P)) = 2^{\nu} \cdot 5=10$. We get the contradiction: $a \in \{20,22\}$ and
$a=2^{\beta}-1+2u_1+u_2=2^{\beta}-1+2+1=2^{\beta}+2$.\\
- If $Q = \sigma(x^{u-1})$, then $a > u-1 = 2m \geq 10$, $P = \sigma(x^{2r})$, $2^{\alpha} = d \leq 3$. Hence, $d = 2$, $2r \leq 4$, $a \in \{4,6,8,10\}$. We get the contradiction: $a > 10 \geq a$.
\end{proof}

\begin{corollary} \label{synthese5}
If $A$ is b.u.p., with $Q$ of the form $\sigma(x^{2m})$, then\\
$\begin{array}{l}
\text{$(P,Q) = (M_1, 1+x(x+1)^{3}{M_1}^{4})$ or $(P,Q)= (M_4, 1+x(x+1){M_4}^{2}),$}\\
\text{$a,b \in \{4,6,8,10,20,22,24,26\}$, $c \in \{1,2,3,7\},\ d=1$}.
\end{array}$
\end{corollary}
\subsubsection{Case where $Q = \sigma(P^{2m})$, for some $m \geq 1$} \label{caseQ=sigmP2r}
Lemma \ref{factorifQ=sigmP2m} implies that $Q, PQ \not\in  \{\sigma(x^{2g}), \sigma((x+1)^{2g}): g \geq 1\}$.
 $P =M_1$ and $(Q =\sigma(P^2) = 1+x(x+1)P$ or
$Q=\sigma(P^4) = 1+x^3(x+1)^3P)$. Thus, $u_1=v_1=1$, $u_2 = v_2 \in \{1,3\}, w_2 = 1$.

We refer to Relations in (\ref{lessigmastaralla}), in (\ref{lessigmastarallb}) and in (\ref{lessigmastar-add3}). Lemma \ref{lesodddivisors} is also useful.
\begin{lemma} \label{lesabcd1}
The integer $a+b$ is odd, $a,b \leq 11$, $c\leq 8$ and $d\leq 3$.
\end{lemma}
\begin{proof}
If $c$ is even, then $2m=2t \geq 2$, $\sigma(P^{2t}) = Q$. So, $w=1, d=1$. If $c$ is
odd, then $Q=\sigma(P^{w-1}), w \in \{3, 5\}, d=2^{\gamma}$.\\
- If $a$ and $b$ are even, then $a, b\geq 4$ (by Lemma \ref{valuesofd}-iv)), $P=\sigma(x^{2r})= \sigma((x+1)^{2s})$.  Hence, $u=v=1$, $2r=2s=2$,
$a,b \leq 6$ and $c=2+d$ (by considering the exponents of $P$).
We get a contradiction on the value of $c$.\\
- If $a$ and $b$ are odd, then $\sigma(x^{u-1}), \sigma((x+1)^{v-1}) \in \{1, P\}$, so that $u,v\leq 3$.\\
Moreover, if $c$ is even, then $\sigma(P^{2t})=Q$, $w=1$, $d=1$ and $c\in \{1,1+2^{\alpha},1+2^{\beta},
1+2^{\alpha}+2^{\beta}\}$. It contradicts the parity of $c$.
If $c$ is odd, then $Q=\sigma(P^{w-1})$, $w \in \{3, 5\}$,
$d=2^{\gamma}$, so that $d=2$ and $c\in \{2,2+2^{\alpha},2+2^{\beta},
2+2^{\alpha}+2^{\beta}\}$. We also get a contradiction on the value of $c$.\\
- If $a$ is even and $b$ odd, then $a \geq 4$ (Lemma \ref{valuesofd}), $\sigma(x^{2r}) =P=M_1$, $u=1$, $2r=2$, $a\leq 6$. Moreover,
$\sigma((x+1)^{v-1}) \in \{1, P\}$, so $v\leq 3$.
We get $\beta \leq 2$, $b\leq 11$, $d \leq 3$ and $c\leq 8$ because $2^{\beta}-1 \leq a\leq 6$, $d \leq a \leq 6$ and $c \in \{1+d,1+2^{\beta}+d\}$.\\
The proof is similar if $a$ is odd and $b$ even.
\end{proof}

\begin{corollary} \label{synthese3}
If $A$ is b.u.p., with $Q$ of the form $\sigma(P^{2m})$, then
$P=M_1$, $Q \in \{1+x(x+1)P, 1+x^3(x+1)^3P\}$,
$a+b$ is odd, $a,b \leq 11$, $c \leq 8, d \leq 3$.
\end{corollary}
\subsubsection{Maple Computations} \label{compute}
{\bf{The function $\sigma^{**}$ is defined as}} Sigm2star, for the Maple code.
\begin{verbatim}
> Sigm2star1:=proc(S,a) if a=0 then 1;else if a mod 2 = 0
then n:=a/2:sig1:=sum(S^l,l=0..n):sig2:=sum(S^l,l=0..n-1):
Factor((1+S)*sig1*sig2) mod 2:
else Factor(sum(S^l,l=0..a)) mod 2:fi:fi:end:
> Sigm2star:=proc(S) P:=1:L:=Factors(S) mod 2:k:=nops(L[2]):
for j to k do S1:=L[2][j][1]:h1:=L[2][j][2]:
P:=P*Sigm2star1(S1,h1):od:P:end:
\end{verbatim}
 We search all $S=x^a(x+1)^b P^cQ^d$ such that $\sigma^{**}(S)= S$.
We apply Corollaries \ref{synthese4},  \ref{synthese5} and \ref{synthese3}.\\
1) If $Q$ or $PQ$ is of the form $\sigma(x^{2m})$, then we obtain no b.u.p. polynomials.\\
2) If $Q$ is of the form $\sigma(P^{2m})$, then we get $D_{1}$,  $D_{2}$, $\overline{D}_1$ and $\overline{D}_2$.
\def\thebibliography#1{\section*{\titrebibliographie}
\addcontentsline{toc}
{section}{\titrebibliographie}\list{[\arabic{enumi}]}{\settowidth
 \labelwidth{[
#1]}\leftmargin\labelwidth \advance\leftmargin\labelsep
\usecounter{enumi}}
\def\newblock{\hskip .11em plus .33em minus -.07em} \sloppy
\sfcode`\.=1000\relax}
\let\endthebibliography=\endlist

\def\biblio{\def\titrebibliographie{References}\thebibliography}
\let\endbiblio=\endthebibliography




\newbox\auteurbox
\newbox\titrebox
\newbox\titrelbox
\newbox\editeurbox
\newbox\anneebox
\newbox\anneelbox
\newbox\journalbox
\newbox\volumebox
\newbox\pagesbox
\newbox\diversbox
\newbox\collectionbox
\def\fabriquebox#1#2{\par\egroup
\setbox#1=\vbox\bgroup \leftskip=0pt \hsize=\maxdimen \noindent#2}
\def\bibref#1{\bibitem{#1}


\setbox0=\vbox\bgroup}
\def\auteur{\fabriquebox\auteurbox\styleauteur}
\def\titre{\fabriquebox\titrebox\styletitre}
\def\titrelivre{\fabriquebox\titrelbox\styletitrelivre}
\def\editeur{\fabriquebox\editeurbox\styleediteur}

\def\journal{\fabriquebox\journalbox\stylejournal}

\def\volume{\fabriquebox\volumebox\stylevolume}
\def\collection{\fabriquebox\collectionbox\stylecollection}
{\catcode`\- =\active\gdef\annee{\fabriquebox\anneebox\catcode`\-
=\active\def -{\hbox{\rm
\string-\string-}}\styleannee\ignorespaces}}
{\catcode`\-
=\active\gdef\anneelivre{\fabriquebox\anneelbox\catcode`\-=
\active\def-{\hbox{\rm \string-\string-}}\styleanneelivre}}
{\catcode`\-=\active\gdef\pages{\fabriquebox\pagesbox\catcode`\-
=\active\def -{\hbox{\rm\string-\string-}}\stylepages}}
{\catcode`\-
=\active\gdef\divers{\fabriquebox\diversbox\catcode`\-=\active
\def-{\hbox{\rm\string-\string-}}\rm}}
\def\ajoutref#1{\setbox0=\vbox{\unvbox#1\global\setbox1=
\lastbox}\unhbox1 \unskip\unskip\unpenalty}
\newif\ifpreviousitem
\global\previousitemfalse
\def\separateur{\ifpreviousitem {,\ }\fi}
\def\voidallboxes
{\setbox0=\box\auteurbox \setbox0=\box\titrebox
\setbox0=\box\titrelbox \setbox0=\box\editeurbox
\setbox0=\box\anneebox \setbox0=\box\anneelbox
\setbox0=\box\journalbox \setbox0=\box\volumebox
\setbox0=\box\pagesbox \setbox0=\box\diversbox
\setbox0=\box\collectionbox \setbox0=\null}
\def\fabriquelivre
{\ifdim\ht\auteurbox>0pt
\ajoutref\auteurbox\global\previousitemtrue\fi
\ifdim\ht\titrelbox>0pt
\separateur\ajoutref\titrelbox\global\previousitemtrue\fi
\ifdim\ht\collectionbox>0pt
\separateur\ajoutref\collectionbox\global\previousitemtrue\fi
\ifdim\ht\editeurbox>0pt
\separateur\ajoutref\editeurbox\global\previousitemtrue\fi
\ifdim\ht\anneelbox>0pt \separateur \ajoutref\anneelbox
\fi\global\previousitemfalse}
\def\fabriquearticle
{\ifdim\ht\auteurbox>0pt        \ajoutref\auteurbox
\global\previousitemtrue\fi \ifdim\ht\titrebox>0pt
\separateur\ajoutref\titrebox\global\previousitemtrue\fi
\ifdim\ht\titrelbox>0pt \separateur{\rm in}\
\ajoutref\titrelbox\global \previousitemtrue\fi
\ifdim\ht\journalbox>0pt \separateur
\ajoutref\journalbox\global\previousitemtrue\fi
\ifdim\ht\volumebox>0pt \ \ajoutref\volumebox\fi
\ifdim\ht\anneebox>0pt  \ {\rm(}\ajoutref\anneebox \rm)\fi
\ifdim\ht\pagesbox>0pt
\separateur\ajoutref\pagesbox\fi\global\previousitemfalse}
\def\fabriquedivers
{\ifdim\ht\auteurbox>0pt
\ajoutref\auteurbox\global\previousitemtrue\fi
\ifdim\ht\diversbox>0pt \separateur\ajoutref\diversbox\fi}
\def\endbibref
{\egroup \ifdim\ht\journalbox>0pt \fabriquearticle
\else\ifdim\ht\editeurbox>0pt \fabriquelivre
\else\ifdim\ht\diversbox>0pt \fabriquedivers \fi\fi\fi
.\voidallboxes}

\let\styleauteur=\sc
\let\styletitre=\it
\let\styletitrelivre=\sl
\let\stylejournal=\rm
\let\stylevolume=\bf
\let\styleannee=\rm
\let\stylepages=\rm
\let\stylecollection=\rm
\let\styleediteur=\rm
\let\styleanneelivre=\rm

\begin{biblio}{99}

\begin{bibref}{BeardBiunitary}
\auteur{J. T. B. Beard Jr}  \titre{Bi-Unitary Perfect polynomials over $GF(q)$}
\journal{Annali di Mat. Pura ed Appl.} \volume{149(1)} \pages 61-68 \annee 1987
\end{bibref}

\begin{bibref}{Canaday}
\auteur{E. F. Canaday} \titre{The sum of the divisors of a
polynomial} \journal{Duke Math. J.} \volume{8} \pages 721-737 \annee
1941
\end{bibref}

\begin{bibref}{Gall-Rahav7}
\auteur{L. H. Gallardo, O. Rahavandrainy} \titre{There is no odd
perfect polynomial over $\F_2$ with four prime factors}
\journal{Port. Math. (N.S.)} \volume{66(2)} \pages 131-145 \annee
2009
\end{bibref}

\begin{bibref}{Gall-Rahav5}
\auteur{L. H. Gallardo, O. Rahavandrainy} \titre{Even perfect
polynomials over $\F_2$ with four prime factors} \journal{Intern. J.
of Pure and Applied Math.} \volume{52(2)} \pages 301-314 \annee 2009
\end{bibref}

\begin{bibref}{Gall-Rahav11}
\auteur{L. H. Gallardo, O. Rahavandrainy} \titre{All unitary perfect
polynomials over $\F_{2}$ with at most four distinct irreducible
factors} \journal{Journ. of Symb. Comput.} \volume{47(4)} \pages 492-502 \annee 2012
\end{bibref}

\begin{bibref}{Gall-Rahav13}
\auteur{L. H. Gallardo, O. Rahavandrainy} \titre{Characterization of Sporadic perfect
polynomials over $\F_{2}$ } \journal{Functiones et Approx.} \volume{55(1)} \pages 7-21 \annee 2016
\end{bibref}

\begin{bibref}{Gall-Rahav14}
\auteur{L. H. Gallardo, O. Rahavandrainy} \titre{All even (unitary) perfect polynomials over $\F_2$ with only Mersenne primes as odd divisors} \journal{Kragujevac J. Math.} \volume{49(4)} \pages 639-652 \annee 2025
\end{bibref}

\begin{bibref}{Gall-Rahav15}
\auteur{L. H. Gallardo, O. Rahavandrainy} \titre{All bi-unitary perfect polynomials over $\F_2$ with only Mersenne primes
as odd divisors} \journal{arXiv:2204.13337v3} \annee 2026
\end{bibref}

\end{biblio}

\end{document}